\documentclass[11pt,a4paper]{article}
\usepackage{amsmath,amsthm,amssymb,amsfonts}
\usepackage{enumitem}   

\usepackage{sectsty}

\sectionfont{\centering}
\subsectionfont{\centering}

\usepackage[svgnames]{xcolor}

\newcommand{\R}{{\mathbb{R}}}          

\newcommand{\XIS}{{\mathfrak{X}}}

\newcommand{\rr}{\rightarrow}
\newcommand{\lrr}{\longrightarrow}

\newcommand{\calf}{{\cal F}}             %
\newcommand{\call}{{\cal L}}             %
\newcommand{\calm}{{\cal M}}             %
\newcommand{\calp}{{\cal P}}             %
\newcommand{\calq}{{\cal Q}}             %
\newcommand{\calu}{{\cal U}}             %

\newcommand{\sg}{{\mathrm{sg}}\,}

\newcommand{\Der}{{\mathrm{Der}}}
\newcommand{\End}[1]{{\mathrm{End}}\,{#1}}

\newcommand{\ip}{{\mathrm{i}}}
\newcommand{\dx}{{\mathrm{d}}}
\newcommand{\inv}[1]{{#1}^{-1}}
\newcommand{\papa}[2]{\frac{\partial#1}{\partial#2}}

\newcommand{\cinf}[1]{{\mathrm{C}}^\infty_{#1}}

\newcommand{\estrela}{{\boldsymbol{\star}}}

\newtheorem{teo}{Theorem}
\newtheorem{prop}{Proposition}

\def\cyclic{\mathop{\kern0.9ex{{+}
\kern-2.2ex\raise-.28ex\hbox{\Large\hbox{$\circlearrowright$}}}}\limits}

\title{The mirror map and other natural tangent tensors}

\author{Rui Albuquerque}

\begin{document}


\maketitle


\begin{abstract}

We review the intrinsic geometry of the tangent bundle of a differentiable manifold $M$, aside from any non-natural structures. We recall the properties of the mirror map $B\in \End{TTM}$, known also as the canonical endomorphism or the almost-tangent structure, solve some cohomological questions and raise other induced by the $\dx_B$ operator.

\end{abstract}


\
\vspace*{5mm}
\vspace*{2mm}

\vspace*{6mm}

\markright{\sl\hfill  Rui Albuquerque \hfill}

\setcounter{section}{1}


\subsection{{Tangent and cotangent manifolds}}



The cotangent bundle of a given differentiable manifold with its natural symplectic structure is a main force of differential geometry. For each $M$ of dimension $m$ and class C$^k$  ($2\leq k\leq\infty$) there is an associated closed non-degenerate 2-form $\dx\lambda$ defined on the total space of the well-known cotangent bundle $\pi:T^*M\lrr M$. One sees it immediately while recalling the differentiable structure of such manifold.

For any chart $(\calu,x^1,\ldots,x^m)$ on an open subset $\calu\subset M$, the coordinates $(x^i,p_j)_{i,j=1,\ldots,m}$, often called position-momentum coordinates, allow us to write the global Liouville 1-form on $T^*M$ as
\[  \lambda= p_j\dx x^j .  \]
One always abbreviates $\dx x^i$ for $\dx(x^i\circ\pi)=\pi^*\dx x^i$, thus giving a first example of a natural lift of a tensor from $\calu$ to $T^*\calu\subset T^*M$. A change of coordinates in $M$, say to
$(\calu',x'^a)_{a=1,\ldots,m}$, yields
\begin{equation}  \label{trancotangent1}
  \dx x'^a=\papa{x'^a}{x^j}\dx x^j\qquad  p'_a=p_j\papa{x^j}{x'^a}  \qquad
\dx p'_a=p_j\papa{^2x^j}{x'^b\partial x'^a}\dx x'^b +\papa{x^j}{x'^a}\dx p_j .
\end{equation}
It follows that $\lambda$ is globally defined and therefore that $\dx\lambda= \dx p_j\wedge\dx x^j$ appears in canonical form, defining the natural symplectic structure on $T^*M$.

Note that whenever $x$ and $x'$ are charts of $M$, ie. homeomorphisms from overlapping open subsets in $M$ into respective open subsets in $\R^m$, then the above notation, representing the euclidean coordinates by $(y^1,\ldots,y^m)$, refers to
\begin{equation} \label{derivationexplained}
 \papa{x'^a}{x^j}(x)=\papa{(x'^a\circ \inv{x})}{y^j}(y^1,\ldots,y^m) .
\end{equation}
Further, using euclidean space, we have
\[ p'_a=p_j\papa{x^j}{x'^a}(x)=p_j\langle\dx(x\circ\inv{x'})_{y'}(e_a),e_j\rangle= \langle(\dx(x\circ\inv{x'}))^\mathrm{ad}_{y'}(p),e_a\rangle . \]

The differentiable structure of a manifold such as $T^*M$ may indeed be defined as the family of open sets $\calu_\alpha\times\R^m$, distinguished by the charts (and these by the $\alpha$) of an atlas of $M$, together with the identification of all its product charts $\calu_\alpha\times\R^m$ on the two products of the intersection $\calu_{\alpha_1}\cap\calu_{\alpha_2}$ with two copies of $\R^m$. On non-empty overlaps, the identification is
\begin{equation}  \label{trancotangent2}
   (x_{\alpha_1},p)\sim(x_{\alpha_2},p') \quad\Leftrightarrow \quad \begin{cases}
  x_{\alpha_1}=x_{\alpha_2} \in \calu_{\alpha_1}\cap\calu_{\alpha_2} \\
    p'=(\dx(x\circ\inv{x'}))^\mathrm{ad}(p)                 \end{cases} ,
\end{equation}
where $x,x'$ denote respectively the chart maps in $\calu_{\alpha_1},\calu_{\alpha_2}$. It follows that $\sim$ defines an equivalence relation and thus a quotient topology on a coset space of the above charts; in this way the cotangent bundle becomes in earnest a Hausdorff second-countable manifold, of class C$^{k-1}$ and dimension $2m$.

We see now that a definition with $p,p'$ replaced by $v,v'$ and $v'=\dx(x'\circ\inv{x})(v)$
in the right hand side of \eqref{trancotangent2}, or \eqref{trancotangent1}, is equally consistent and therefore a case to follow. It yields a new Hausdorff second-countable topological space $TM$ and henceforth a differentiable manifold, which is well known as the tangent bundle of $M$. Moreover, for each fixed point $x\in M$, invariantly of the choice of the chart, we may speak of the obvious $m$-dimensional vector spaces $T^*_xM$ and $T_xM$, also embedded submanifolds, and indeed understand $T^*M$ as a linear dual of $TM$.

In the search for natural tensors, the structure of the tangent manifold $TM$ can be considered in the same footing as that of $T^*M$, a {sublime} duality leading to fruitful results rather than to formal symmetry. For instance, we observe that the total space $TM$ does \textit{not} have a natural symplectic structure.

Having recalled the two differentiable structures, we just assume from now on the manifold is of class C$^\infty$ or smooth.

Let us denote also by $\pi$ the bundle projection for the tangent bundle $\pi:TM\lrr M$ of the given $m$-dimensional manifold $M$. Here the charts $(\inv{\pi}(\calu),x^i,v^j)$ and $(\inv{\pi}(\calu'),x'^a,v'^b)$, induced by charts on $M$ and by the respective euclidean coordinates are related by
\begin{equation}  \label{trantangent1}
   \dx x'^a=\papa{x'^a}{x^j}\dx x^j\qquad  v'^a=v^j\papa{x'^a}{x^j}  \qquad
\dx v'^a=v^j\papa{^2x'^a}{x^i\partial x^j}\dx x^i+\papa{x'^a}{x^j}\dx v^j  .
\end{equation}

At this point we recall that if for one atlas of $M$ all $\papa{x'^a}{x^i}$ are constant then we say we have a \textit{flat} structure on $M$. It considerably simplifies the above equations.

$TM$ is always an orientable manifold, though not carrying a preferred top degree form. Indeed, $\dx x'^1\wedge\cdots\wedge\dx x'^m\wedge\dx v'^1\wedge\cdots\wedge\dx v'^m$ transform on chart overlaps by a positive factor: the square of $\det(\partial x'^a/\partial x^j)$.

Coordinate vector fields on $TM$ verify
\begin{equation}   \label{trantangent2}
  \papa{}{x'^a}= \papa{x^i}{x'^a}\papa{}{x^i}+v'^b\papa{^2x^j}{x'^a\partial x'^b}\papa{}{v^j}
\quad\qquad  \papa{}{v'^a}=\papa{x^i}{x'^a}\papa{}{v^i} ,
\end{equation}
transforming coherently with any third chart, as required. To see the latter, we use
\begin{equation}    \label{equacaoutil}
   v'^b\papa{x'^a}{x^k}\papa{^2x^j}{x'^a\partial
x'^b}+v^l\papa{x^j}{x'^a}\papa{^2x'^a}{x^k\partial x^l} =\papa{v^j}{x^k}=0 .
\end{equation}

There is a well-defined ring $\oplus_k\calp_k$ of sums of \textit{degree} $k$ \textit{polynomial functions along the fibre}, consisting of the smooth $f:TM\lrr\R$ such that $\papa{^{k+1}f}{v^{i_1}\cdots\partial v^{i_{k+1}}}=0$, for all sets of indices, and having a non-vanishing order $k$ derivative in some chart.

Sections of $T^*M$ over $M$ are called \textit{covariant fields} or 1-\textit{forms} on $M$ while sections of $TM$ are the \textit{contravariant} or \textit{vector fields} on $M$. The respective $\cinf{M}$-modules of sections are denoted by $\Omega^1_M$ and $\XIS_M$, as it is well-known\footnote{For a matter of coherence of this text, we recall here that any vector field $X$ induces a flow of local diffeomorphisms of $M$ and then a Lie bracket operator $\call_X:\XIS\rr\XIS$, which gives $\XIS$ a Lie algebra structure. From 1-forms we can produce $p$-forms $\omega\in\Omega^p_M$ algebraically on the manifold, ie. via exterior algebra. We may further differentiate these tensors, either through exterior derivative operator $\dx$, through Lie derivative $\call_X$ along the flow of $X$ or through interior product $\ip_X$. The three satisfy a Leibinz rule on the so called wedge product, with or without signs attached. For the Lie derivative applied to $p$-forms there are no signs. Finding the derivatives on 0,1 and 2-forms, one proves by induction a famous Cartan formula: $\call_X\omega=\dx \ip_X\omega+\ip_X\dx\omega$.}.

We may write at each point $(x^i,v^j)$ the lifts
\begin{equation} \label{levantamentotrivial}
  \papa{}{x^i}=\partial_{x^i}\qquad\qquad\papa{}{v^i}=\partial_{v^i} =\pi^\star\partial_{x^i} .
 \end{equation}
Under the {vertical lift} $\pi^\star$, indeed every $X\in\XIS_M$ lifts globally to $\pi^\star X\in\Gamma(TM;V)$, ie. a vector field tangent to the fibres: if locally $X=X^i\papa{}{x^i}$, then
\begin{equation}  \label{levantamentovertical}
   \pi^\star X=X^i\papa{}{v^i}
\end{equation}
is a well-defined global vertical vector field on $TM$. Notice $X^i$ stands in place of $\pi^*X^i=X^i\circ\pi$. Moreover, the vertical lift is tensorial, we may extend it to $(0,q)$-tensors. Notice throughout the text that $\pi^*$ and $\pi^\star$ shall have very precise meanings.

It is trivial to describe the vector bundle epimorphism $\dx\pi:TTM\lrr\pi^*TM$ in coordinates, as well as the inclusion map $\pi^\star TM\,\subset TTM$, since locally we have frames of the type
$\papa{}{x^1},\ldots,\papa{}{x^m},\papa{}{v^1},\ldots,\papa{}{v^m}$.

The {vertical bundle} is the vector subbundle $V=\ker\dx\pi=\pi^\star
TM\,\subset TTM$. Locally $V$ is spanned by the $\partial_{v^i}$ 

Notice that $p$-forms also have a natural lift, which is the pull-back. Finally, the tensors of type $(p,q)$ on $M$ admit what is called the vertical lift (giving precedence to the contravariant side): given $\alpha$ of type $(p,0)$ and $Q$ of type $(0,q)$, we define
\begin{equation}\label{tensoredverticallift}
   (\alpha\otimes Q)^\star=\pi^*\alpha\,\otimes \pi^\star Q .
\end{equation}


\subsection{Complete lifts}


A distinct property of tangent bundles seems to be that of the complete lift of a vector field on $M$ to $TM$ or, indeed, the complete lift of any tensor on $M$ to $TM$. It was almost a century ago that the notion of extended or complete lift was introduced in the theory of foliations. To the best of our knowledge, the relation of the extended vector fields with tangent manifolds occurs first in \cite{Sasa}. There, Sasaki also defines the extension of 1-forms, which are later thoroughly treated in \cite{IshiharaYano}. We recall here some of those results.

Given $X\in\XIS_M$, the \textit{extended} or \textit{complete lift} of $X$ to a vector field on $TM$ is given in a standard chart by
\begin{equation}  \label{levantamentocompleto}
   \widetilde{X}=X^i\papa{}{x^i}+v^j\papa{X^i}{x^j}\papa{}{v^i} .
\end{equation}
This is a well-defined global vector field; to see this, one uses again identity \eqref{equacaoutil}.

Vertical lifts are linearly independent of complete lifts. Since $\widetilde{\papa{}{x^i}}=\papa{}{x^i}$, the complete lifts together with the vertical lifts of a coordinate frame of $M$ become a frame of $TM$. Moreover, this is true for any given frame on $M$.


The complete lift $\widetilde{X}$ of a vector field $X\in\XIS_M$ is the 1st-order part of the differential map of a given diffeomorphism of $M$. One verifies that
\begin{equation}
 X=\frac{\dx}{\dx t}\phi_t(x) \quad\Leftrightarrow\quad \widetilde{X}=\frac{\dx}{\dx t}\dx\phi_t(x,v) .
\end{equation}
The following identities hold for any vector fields $X,Y$ on $M$:
\begin{equation}\label{parentesisdeLieentrecanonicos}
 [\widetilde{X},\widetilde{Y}]=\widetilde{[X,Y]},\qquad [\widetilde{X},\pi^\estrela Y]=\pi^\estrela[X,Y],\qquad
[\pi^\estrela X,\pi^\estrela Y]=0.
\end{equation}

The lift of a function $f$ on $M$ to a function on $TM$ is obvious. We use the notation $\pi^*f$ only if necessary. These functions also called the constant along the fibre functions.

The \textit{complete lift} of a function $f\in \mathrm{C}^k_M$ is given by $\widetilde{f}$ the function on $TM$ defined by
\[ \widetilde{f}(x,v)=\dx f(x^i,v^j)=v^j\papa{f}{x^j}.\]
In particular, $\widetilde{x^k}=v^k$.
The operator is linear on functions and satisfies an obvious Leibniz rule: $\widetilde{fg}=\widetilde{f}\,g+f\,\widetilde{g}$.

Given a $p$-form $\alpha\in\Omega^p_M$, recall the pull-back  $\pi^*\alpha$ is a $p$-form on $TM$, which vanishes on any vertical directions\footnote{The pull-back is called the vertical lift in \cite{IshiharaYano}.}. For any $X,Y \in\XIS_M$ and any $f\in\mathrm{C}^k_M$ constant along the fibre,
\[ \qquad (\pi^\star X)\cdot f=0,\qquad \pi^*\alpha(\pi^\star Y)=0,\qquad
                                  \widetilde{\partial_{x^i}}=\partial_{x^i}, \]
\[  \widetilde{X}\cdot f=\pi^*(X\cdot f) \qquad \pi^*\alpha(\widetilde{X})=\pi^*(\alpha(X)) . \]
\[  \widetilde{X}\cdot\widetilde{f}=\widetilde{X\cdot f}\qquad\qquad
\widetilde{fX}=f\widetilde{X}+\widetilde{f}\pi^\star X . \]

Notice the tensoriality of the following map defined on $\XIS_M^2$ by
\[  \widetilde{X\otimes Y}=\widetilde{X}\otimes \pi^\star Y+\pi^\star X\otimes\widetilde{Y} . \]

One also defines the \textit{complete lift} of a 1-form $\alpha\in\Omega^1_M$. Indeed, in a chart $(x^i,v^j)$, with $\alpha=a_i\dx x^i$,
\begin{equation}\label{completelift_oneforms}
 \widetilde{\alpha}=v^j\papa{a_i}{x^j}\dx x^i+a_i\dx v^i.
\end{equation}
Once again, \eqref{equacaoutil} proves this definition is independent of the choice of chart. Since $\pi$ is a submersion, both pull-back and complete lift operators are injective.

The following identities are satisfied:
\[  \widetilde{\alpha}(\pi^\star X)=\pi^*(\alpha(X))= \pi^*\alpha(\widetilde{X}) \qquad\quad
\widetilde{\alpha}(\widetilde{X})=\widetilde{\alpha(X)}  \]
\[  \dx\widetilde{f}=\widetilde{\dx f}   \]
\[   \widetilde{f\alpha}=\widetilde{f}\pi^*\alpha+f\widetilde{\alpha} .  \]

Due to three Leibniz rules involving the complete lift, we may define more generally the \textit{complete lift} of a tensor product, for any $P,Q$ tensors on $M$, by, cf. \eqref{tensoredverticallift},
\begin{equation}  \label{completelift_tensorial}
 \widetilde{P\otimes Q}= \widetilde{P}\otimes Q^\star+P^\star\otimes\widetilde{Q} .
\end{equation}
For example, $\widetilde{1_M}=(\dx x^i\otimes\partial_{x^i})^\sim=\dx v^i\otimes\partial_{v^i}+\dx x^i\otimes\partial_{x^i}=1_{TM}$.

It follows that for $\alpha\in\Omega^p_M,\beta\in\Omega^q_M$, we have the $(p+q)$-form on $TM$
\begin{equation}
\widetilde{\alpha\wedge\beta}= \widetilde{\alpha}\wedge\pi^*\beta+\pi^*\alpha\wedge\widetilde{\beta} .
 \end{equation}
Admitting  $\widetilde{\dx\alpha}=\dx\widetilde{\alpha}$, for $p,q$-forms, and recalling $\dx\pi^*=\pi^*\dx$, we see as an induction step that
\begin{equation*}
 \begin{split}
  \dx (\widetilde{\alpha\wedge\beta}) & =
   \widetilde{\dx\alpha}\wedge\pi^*\beta+(-1)^p\widetilde{\alpha}\wedge\pi^*{\dx\beta}
  + \pi^*{\dx\alpha}\wedge\widetilde{\beta}+(-1)^p\pi^*\alpha\wedge\widetilde{\dx\beta}\\
  & =\widetilde{\dx\alpha\wedge\beta}+(-1)^p\widetilde{\alpha\wedge\dx\beta} \\
  & =({\dx(\alpha\wedge\beta) })^\sim
 \end{split}
\end{equation*}
and therefore, for any differential form, $\widetilde{\dx\alpha}=\dx\widetilde{\alpha}$.

All the above results are deduced in \cite{IshiharaYano}.

\begin{prop} \label{Lemmaanulamentolevantamentocompleto}
\begin{enumerate}[label=(\roman*)]
 \item For any function $f$ on $M$, we have $\widetilde{f}$ is a constant iff $\widetilde f=0$ iff $f$ is a constant.
 \item For any $p\geq1$ and $\omega\in\Omega^p_M$ we have $\widetilde{\omega}=0$ iff $\omega=0$.
\end{enumerate}
\end{prop}
\begin{proof}
For (i) one evaluates $\widetilde{f}=v^j\papa{f}{x^j}$ at a canonical basis associated to a chart. For
(ii), writing $\omega=o_{i_1\cdots i_p}\dx x^{i_1}\wedge\cdots\wedge\dx x^{i_p}$ with $i_1<i_2<\cdots<i_p$, we have
 \begin{equation*}
  \begin{split}
    \widetilde{\omega} =v^j\papa{o_{i_{i_1\cdots i_p}}}{x^j}\dx x^{i_1}\wedge\cdots\wedge\dx x^{i_p}+ \hspace{53mm} \\
    o_{i_1\cdots i_p}(\dx v^{i_1}\wedge\dx x^{i_2}\wedge\cdots\wedge\dx x^{i_p}+\dx x^{i_1}\wedge\dx v^{i_2}\wedge\cdots\wedge\dx x^{i_p}+\cdots+\dx x^{i_1}\wedge\cdots\wedge\dx x^{i_{p-1}}\wedge\dx v^{i_p})
  \end{split}
 \end{equation*}
 and the result follows.
\end{proof}

\subsection{Cohomology with complete lifts}

The cohomology space $H^p(M,\R)=\{\alpha\in\Omega^p_M:\ \dx\widetilde{\alpha}=0 \}/\{ \dx\gamma: \ \gamma\in\Omega^{p-1}_M\}$ can be thought of coherently, since $\dx(\dx\gamma)^\sim=(\dx\dx\gamma)^\sim=0$. However, proposition \ref{Lemmaanulamentolevantamentocompleto} implies it is just the de\,Rham cohomology of $M$, as notation sugests.

We have another cohomology arising from the complete lifts $\Omega^p_M\hookrightarrow\Omega^p_{TM}$.
\begin{prop}
 For any $0\leq p\leq m$, there exists a linear epimorphism induced by the complete lift
 \begin{equation}
  H^p(M,\R)\lrr H^p(TM)^\sim:=\frac{\{\widetilde{\alpha}:\ \alpha\in\Omega^p_M,\ \,\dx\alpha=0\}}{\{\dx\beta:\ \beta\in\Omega^{p-1}_{TM} \}}.
 \end{equation}
 Moreover, $H^0(TM)^\sim=H^1(TM)^\sim=0$. The morphism is not injective in general.
\end{prop}
\begin{proof}
 We have seen the map is well defined. Of course it is surjective. For the case $p=1$, suppose the form $\alpha=a_i\dx x^i$ on $M$ is closed. Hence $\partial_{x^i}a_j=\partial_{x^j}a_i$. We have $\widetilde{\alpha}=v^j\papa{a_i}{x^j}\dx x^i+a_i\dx v^i$ and $\widetilde{\alpha}=\dx F$ is equivalent to
 \[  \papa{F}{x^i}=v^j\papa{a_i}{x^j} \qquad\ \papa{F}{v^i}=a_i \]
 Immediately we find $F=a_iv^i+c$ where $c$ is a constant.
\end{proof}
Let us see an example. Let $M=\R^2\backslash\{(0,0)\}$ have coordinates $(x,y)$ and $TM=M\times\R^2$ coordinates $(x,y,v,w)$. The well-known 1-form $\omega=\frac{-y\dx x+x\dx y}{x^2+y^2}$ is closed and not exact; likewise, $\widetilde{\omega}$ on $TM$ is exact:
\[ \widetilde{\omega}=\dx\bigl(\frac{xw-yv}{x^2+y^2}\bigr) . \]

It is worthless searching for any case of a non-vanishing $H^p(TM)^\sim\subset H^p(TM)$ with dimension strictly less than $\dim H^p(M)$. The question is solved below.

\subsection{The tautological vector field}

Let us introduce one of the natural tensors on $TM$, the so-called \textit{tautological}, \textit{canonical} or \textit{Liouville} vector field $\xi$. At each point $u=(x^i,v^j)$,
\begin{equation}
   \xi=\sum_{i=1}^m  v^i\papa{}{v^i} .
\end{equation}
We may write simply  $\xi_u=\pi^\star u$ for this globally defined vertical vector field. In the context of vector bundle theory some authors call $\xi$ the Euler vector field.

\begin{prop}
For any $p$-form $\alpha$ on $M$:
\begin{equation}
 \mbox{(i)}\ \:\xi\lrcorner\pi^*\alpha=0 \qquad\qquad\qquad
 \mbox{(ii)}\ \:\call_\xi\pi^*\alpha=0
\end{equation}
\end{prop}
\begin{proof}
 (i) is clear since $\xi\lrcorner\pi^*\alpha(X_1,\ldots,X_{p-1})= \pi^*\alpha(\xi,X_1,\ldots,X_{p-1})=0$ for any vectors $X_1,\ldots,X_{p-1}\in TTM$ and  $\xi$ lies in $\ker\dx\pi$. For (ii) we recall formulae $\dx\pi^*\alpha=\pi^*\dx\alpha$ and that of Cartan, giving $\call_\xi\pi^*\alpha=\dx(\xi\lrcorner\pi^*\alpha)+\xi\lrcorner\pi^*\dx\alpha=0$.
 \end{proof}
 The following result is original and surprising thus we write it separately.
 \begin{teo} \label{teo_Lieisidentitty}
  For any $p$-form $\alpha$ on $M$:
  \begin{equation}
    \call_\xi\widetilde{\alpha}=\widetilde{\alpha}.
  \end{equation}
In particular, every closed form $\alpha$ on $M$ yields an exact form $\widetilde{\alpha}=\dx(\xi\lrcorner\widetilde{\alpha})$ on $TM$. And thus, for all $0\leq p\leq m$,
 \[  H^p(TM)^\sim=0. \]
 \end{teo}
\begin{proof}
 We prove first the case $p=0$. For any smooth function $f$ on $M$, $\widetilde{f}=v^j\papa{f}{x^j}$, and so $\call_\xi\widetilde{f}=\xi\lrcorner\dx\widetilde{f}=\xi\lrcorner(v^j\papa{^2f}{x^ix^j}\dx x^i+\papa{f}{x^j}\dx v^j)=\widetilde{f}$. And we prove for $p=1$. Let $\alpha=a_i\dx x^i$, so that we have immediately $\widetilde{\alpha}=v^k\papa{a_i}{x^k}\dx x^i+a_i\dx v^i$ and
\[ \dx\widetilde{\alpha} =v^k\papa{^2a_i}{x^kx^j}\dx x^j\wedge\dx x^i+\papa{a_i}{x^j}(\dx v^j\wedge\dx x^i+\dx x^j\wedge\dx v^i). \]
Hence $\xi\lrcorner\dx\widetilde{\alpha}=v^k\papa{a_i}{x^k}\dx x^i-v^i\papa{a_i}{x^j}\dx x^j$.
We see also that $\xi\lrcorner\widetilde{\alpha}=v^ia_i$ and therefore $\dx(\xi\lrcorner\widetilde{\alpha})=v^i\papa{a_i}{x^j}\dx x^j+a_i\dx v^i$. Combining the two we find $\call_\xi\widetilde{\alpha}=a_i\dx v^i+v^k\papa{a_i}{x^k}\dx x^i=\widetilde{\alpha}$. We see locally for a product of $p,q$-forms
\begin{equation*}
 \begin{split}
  \call_\xi\widetilde{\alpha\wedge\beta}&=\call_\xi(\widetilde{\alpha}\wedge\pi^*\beta+\pi^*\alpha\wedge\widetilde{\beta})\\
  &= \call_\xi\widetilde{\alpha}\wedge\pi^*\beta+ \pi^*\alpha\wedge\call_\xi\widetilde{\beta} \\
  &= \widetilde{\alpha}\wedge\pi^*\beta+\pi^*\alpha\wedge\widetilde{\beta}=\widetilde{\alpha\wedge\beta}
 \end{split}
\end{equation*}
and the result follows by induction.
\end{proof}

 \subsection{The mirror map}


In previous works the author has given the name of $B$-{\textit{field}} or {\textit{mirror map}} to the global tensor of type $(1,1)$ defined in any chart by
\[   B=\sum_{j=1}^m \dx x^j\otimes \papa{}{v^j} .  \]
We can only speculate which tensor $B$ or $\xi$ defined on $TM$ is the dual of the Liouville 1-form on the cotangent manifold.

The tensor $B$ is recurrent in the last decades research thriving between differential geometry and Hamiltonian mechanics, cf. for example \cite{CGGM,Crampin1,FLMMR} and their bibliography. This is just as one would expect. $B$ is called the `vertical endomorphism' in \cite{CGGM,LecandaRoy} or the  `canonical endomorphism' in \cite{LecandaRoy,LeonLAinzGordonMarrero} or `the vector 1-form defining the tangent structure' in \cite{AnonaRatovoarimanana}, just to show a few recent appearences. Interesting enough, the tensor $B$ is also an example of a null structure as defined in \cite{Dunajski}.

The first ever to use the mirror map seemed, on a first notice, to have been \cite{Klein}. And then again quite explicitly in \cite{KleinVoutier}, where it was called the `natural almost-tangent structure on $TM$'. Early studies focused on the almost-tangent structure are worth mentioning, such as \cite{DaviesYano,Eliopoulos,Grifone}, interested in the theory of sprays and connections and metric structures.

However, we also found that the notion of almost-tangent structure was established  before in \cite{ClaBruck2}: a manifold endowed with a constant rank (1,1) tensor with vanishing square. Such authors were well aware of an \textit{almost-tangent structure}, precisely the $B$-field or mirror map,  as a particular case appearing over the tangent manifold of any given manifold. This is clearly seen in \cite{ClaBruck1}.

The present author recurred to the mirror endomorphism independently, and has coined this name more than a decade ago in the context of a manifold endowed with an affine connection. He found later that the connection is not required at all for the definition of $B$, cf. \cite{Alb2019b} and the references therein. The affine connection overshadows the notion of the mirror map as a natural object, even though turning it more obvious.

Finally we note that any \textit{horizontal} lift $X^h$ of any given $X\in\XIS_M$ is mirrored through $B$ to the respective vertical lift $B(X^h)=\pi^\estrela X$. Hence the name for $B$. It still seems more likely to the author to use a name for an object which is truly universal, rather than the expressions found in the literature.

\begin{prop}
 The tautological vector field $\xi$ and the mirror map $B$ are natural global, respectively, $(0,1)$- and $(1,1)$-tensor fields defined on $TM$.

 Neither $\xi$ or $B$ are complete lifts of tensors on $M$.
\end{prop}

The endomorphism $B:TTM\lrr TTM$ is also given by $B=\pi^\estrela\circ\pi_*$, but $\pi^\estrela$ is not a bundle map. $B$ verifies
\[ B(TTM)=V=\ker\dx\pi=\ker B\ \ \qquad\ \ \ B\circ B=0. \]
\begin{prop}\label{propdeB}
 We have $B(\widetilde{X})=\pi^\estrela X$ and
 \[ \call_{\pi^\estrela X}B=0\qquad\qquad \call_\xi B=-B \qquad\qquad \call_{\widetilde{X}}B=0 ,\ \ \forall X\in\XIS_M .\]
\end{prop}
\begin{proof}
The first identity is obvious. Now, $\call_WB=\call_W\dx x^j\otimes \partial_{v^j}+\dx x^j\otimes\call_W\partial_{v^j}$ is easy to deduce in most general terms. Giving in particular $\call_{\pi^\estrela X}B=0$ by application of \eqref{parentesisdeLieentrecanonicos}. And also the other identities. Notice that $\call_\xi\partial_{v^j}=[\xi,\partial_{v^j}]=-\partial_{v^j}$. Also $\call_\xi\dx x^j=\dx(\xi\lrcorner\dx x^j)=0$. Then
 \[ \call_\xi B=\call_\xi\dx x^j\otimes\partial_{v^j}+\dx x^j\otimes\call_\xi\partial_{v^j}=-B . \]
 Finally, we have $\call_{\widetilde{X}}\partial_{v^j}=-\papa{X^i}{x^j}\partial_{v^i}$ and $\call_{\widetilde{X}}\dx x^j=\dx(\widetilde{X}\lrcorner\dx x^j)=\papa{X^j}{x^k}\dx x^k$, hence the result.
\end{proof}

As the author has been \textit{proving} since long ago, any set of endomorphisms $B_1,\ldots,B_p$ of a tangent bundle allows us to construct new differential forms from a given differential form $\eta\in\Omega^p$. 
Let us recall that on a product $\eta=\eta_1\wedge\cdots\wedge\eta_p$ of 1-forms, we defined in \cite{Alb2019b}
\begin{equation} \label{defini_etacircwedge}
   \eta\circ(B_1\wedge\cdots\wedge B_p)=\sum_{\sigma\in \mathrm{S}_p} \eta_1\circ B_{\sigma_1}\wedge\cdots\wedge\eta_p\circ B_{\sigma_p} .
\end{equation}
This extends linearly to any $\eta\in\Omega^p$.

A typical application is given in the following result, which motivates this article, in the search for new differential forms.
\begin{prop} \label{nogoProp}
 For any $p$-form $\gamma$ on $M$ and any $0\leq i \leq p$, we have
\begin{equation}
   \frac{1}{p!}\,\widetilde{\gamma}\circ(B^{\wedge,i}\wedge1^{\wedge,p-i})=\begin{cases}
         0 & \mbox{if}\ i\geq2 \\
       \pi^*\gamma &\mbox{if}\ i=1\\
       \widetilde{\gamma} & \mbox{if}\ i=0   .        \end{cases}
\end{equation}
\end{prop}
\begin{proof}
We may assume $\gamma$ is decomposable as above. Then
\begin{eqnarray*}
 \widetilde{\gamma} &=& \widetilde{\gamma_1}\wedge\pi^*(\gamma_2\wedge\cdots\wedge\gamma_p)+\pi^*\gamma_1\wedge(\widetilde{\gamma_2\wedge\cdots\wedge\gamma_p}) \\
 &=&\sum_j\pi^*\gamma_1\wedge\cdots\wedge\widetilde{\gamma_j} \wedge\cdots\wedge
 \pi^*\gamma_p
\end{eqnarray*}
and so
\[ \widetilde{\gamma}\circ(B^{\wedge,i}\wedge1^{\wedge,p-i}) = \sum_j\sum_{\sigma\in\mathrm{S}_p}
 \pi^*\gamma_1\circ B_{\sigma_1}\wedge\cdots\wedge\widetilde{\gamma_j}\circ B_{\sigma_j}\wedge\cdots\wedge\pi^*\gamma_p\circ B_{\sigma_p}
\]
where $B_1=\ldots=B_i=B$ and $B_{i+1}=\ldots=B_p=1$. Now if $i\geq2$, for each $j$ and $\sigma$, some $k\neq j$ will satisfy $\pi^*\gamma_k\circ B_{\sigma_k}=\pi^*\gamma_k\circ B=\gamma_k\circ\pi_*\circ B=0$.

We also note that $\widetilde{\gamma_j}\circ B=\pi^*\gamma_j$. Hence, if $i=1$,
the computation above gives
\begin{eqnarray*}
\widetilde{\gamma}\circ(B^{\wedge,i}\wedge1^{\wedge,p-i}) &=& \sum_{j=1}^p\sum_{\sigma\in\mathrm{S}_p,\,\sigma_j=1}
 \pi^*\gamma_1\circ B_{\sigma_1}\wedge\cdots\wedge\widetilde{\gamma_j}\circ B_{\sigma_j}\wedge\cdots\pi^*\gamma_p\circ B_{\sigma_p} \\
 &=& \sum_j\sum_{\sigma\in\mathrm{S}_p,\,\sigma_j=1} \pi^*\gamma_1\wedge\cdots\wedge\pi^*\gamma_j\wedge\cdots\pi^*\gamma_p \,=\, p!\,\pi^*\gamma.
\end{eqnarray*}
Finally if $i=0$, then the above becomes
\[ \sum_j\sum_{\sigma\in\mathrm{S}_p} \pi^*\gamma_1\wedge\cdots \wedge\widetilde{\gamma_j}\wedge\cdots\wedge\pi^*\gamma_p \,=\, p!\,\widetilde{\gamma}       \]
and the result follows.
\end{proof}
Finding new $p$-forms with a complete lift and the usual procedure is thus overruled.

The following concepts are considered by the author in \cite{Alb2019a}. We call a vector field $W$ over $TM$ a $\lambda$-\textit{mirror vector field} if
\[ \call_WB=\lambda B \]
for some function $\lambda$ on $TM$. Proposition \ref{propdeB} gives some examples of $\lambda$-mirror vector fields.

Let us refer here the theory of sprays, developed also by \cite{Grifone,Klein}.
A spray is a vector field $S$ on $TM$ such that $BS=\xi$. In other words, $\dx\pi_u(S)=u,\ \forall u\in TM$. Sprays become important objects capable of inducing connections and special Lie algebras of vector fields on the tangent manifold, cf. \cite{Anona1}.

\subsection{A mirror of the first cohomology}

Given any $p$-form $\mu$ on $M$, we may consider the $p$-forms ${\alpha}_\mu$ defined on $TM$ such that ${\alpha}_\mu\circ B^{\wedge,p}=\pi^*\mu$. The sum of two such forms for two base $\mu_1,\mu_2$ is a form of the same type. Indeed, any $\alpha_{\mu_1},\alpha_{\mu_2}$ and any $c\in\R$  give us a $p$-form $\alpha_{\mu_1}+c\alpha_{\mu_2}$ of the type $\alpha_{\mu_1+c\mu_2}$.

Those forms exist at least for $p=1$. We may take for instance $\widetilde{\mu}$; this is a type $\alpha_\mu$ form by Proposition \ref{nogoProp}.

The theory of connections allows us to define one $\alpha_\mu$ for any given $\mu$, in a canonical fashion, but we shall not take that path here.

Any forms \textit{above} the same $\mu$ differ by some form $\alpha_0$, this is, $\alpha_0\circ B^{\wedge,p}=0$, or in other words, by some $p$-form which does not have the component $\dx v^{j_1}\wedge\cdots\wedge\dx v^{j_p}$ corresponding to the component $\dx x^{j_1}\wedge\cdots\wedge\dx x^{j_p}$ in $\mu$. The pull-back forms are type $\alpha_0$.

Let us see an example. On $T\R^2$, the $\alpha_0$ 2-forms are written
$\alpha_0=a_{00}\dx x^1\wedge\dx x^2+\sum_{i,j=1}^2 a_{ij}\dx x^i\wedge\dx v^j$
where $a_{00},a_{ij}$ are functions of $(x^1,x^2,v^1,v^2)$. On the other hand, no 2-form $\mu$ on $\R^2$ exists such that  $v^1\dx v^1\wedge\dx v^2$ is a $\alpha_\mu$.

In general we have a proper real vector subspace $\calf^p\subset\Omega^p_{TM}$ of \textit{above} $p$-forms, for every $p\leq n$. There is no immediate coboundary operator though, for a complex sequence $\calf^*$; but for $p=1$ something is worth mentioning.

For a 1-form $\mu\in\Omega_M$, let us suppose a $\alpha_\mu$ exists as above. Locally $\mu=a_i\dx x^i$ on $\calu\subset M$ and hence $\alpha_\mu=h_i\dx x^i+a_i\dx v^i$, with $h_i$ functions on $\inv{\pi}(\calu)$. We have
\begin{equation}\label{closed_alpha_mu_oneforms}
 \dx\alpha_\mu=0 \qquad\Longleftrightarrow\qquad\begin{cases} \papa{h_i}{x^j}=\papa{h_j}{x^i}  \\  \papa{h_j}{v^i}=\papa{a_i}{x^j} \end{cases}\ .
\end{equation}
For any 1-form $\mu\in\Omega_M$, we may consider the functions on $TM$ of the form
\begin{equation} \label{affinefunctions}
  f_\mu(x,v)= a_i(x)v^i +c(x)
\end{equation}
ie. $f_\mu=\xi\lrcorner\pi^*\mu + c$ where $c$ is any function on $M$. It follows that $f_\mu$ is well defined on $TM$, and its differential
\[ \dx f_\mu = (\papa{a_j}{x^i}v^j+\papa{c}{x^i})\dx x^i+a_i\dx v^i\] 
induces an $\alpha_\mu$.
\begin{prop} \label{affinefunctions_equivalences}
For any  $f\in\cinf{TM}$ the following are equivalent:\,\ (i) $f=f_\mu$ for some 1-form $\mu\in\Omega_M$;\,\ (ii) $f$ induces an $\alpha_\mu=\dx f$ for some 1-form $\mu$;\,\ (iii) $f$ verifies $\papa{^2f}{v^i\partial v^j}=0$ (on one and hence any chart).
\end{prop}
\begin{proof}
 (i)$\Rightarrow$(ii) From \eqref{affinefunctions} we see immediately that $ \dx f_\mu\circ B=\pi^*\mu$.\ \,
(ii)$\Rightarrow$(iii) If $f$ satisfies $\dx f\circ B=\pi^*\mu$ then since $\dx f=\papa{f}{x^i}\dx x^i+\papa{f}{v^i}\dx v^i$ it follows that $\papa{f}{v^i}=a_i$ where $a_i$ are the components of $\mu$. Hence $\papa{^2f}{v^i\partial v^j}=0$.\ \,
(iii)$\Rightarrow$(i) Vanishing second derivative on vertical directions implies $f(x,v)=a_i(x)v^i+c(x)$ on each $T_xM$. Taking another chart $(x'^p,v'^q)_{p,q=1,\ldots,m}$, we have
\[  a'_p=\papa{f}{v'^p}=\papa{x^i}{x'^p}\papa{f}{v^i}=\papa{x^i}{x'^p}a_i \]
proving that $\mu=a_i\dx x^i$ is globally defined.
\end{proof}
Now the question is to determine the `closed $\calf^1$ forms mod \textit{any} exact 1-forms'.
 \begin{teo}
 We have
 \begin{equation}
  \frac{\{\alpha_\mu:\ \dx\alpha_\mu=0,\ \mu\in\Omega^1_M\}}{\{\alpha_\nu:\ \alpha_\nu=\dx f_\nu,\ \nu\in\Omega_M^1,\ c\in \cinf{M}\}}\simeq H^1(M;\R).
 \end{equation}
\end{teo}
\begin{proof}
Let us take a closed 1-form $\alpha_\mu=h_i\dx x^i+a_i\dx v^i$ where $\mu=a_i\dx x^i$. The isomorphism we are referring is induced from $\Theta$, which is described by
 \[   \Theta(\alpha_\mu)=\gamma \]
with $\gamma$ given by $\gamma=g_i\dx x^i$ where $g_i=h_i-\papa{a_k}{x^i}v^k$. Notice $\papa{g_i}{v^j}=0$ by \eqref{closed_alpha_mu_oneforms} so $\gamma$ descends to $M$. Moreover, $\gamma$ is globally defined, as the reader may care to prove that $g'_p=\papa{x^i}{x'^p}g_i$ transforms the way it should. Finally, the first half of system \eqref{closed_alpha_mu_oneforms} and together with Schwarz Theorem proves that $\gamma$ is closed.

 The linear map $\Theta$ is well defined on cohomology since, if $\alpha_\mu=\dx f$, then
$\alpha_\mu=\dx f_\nu$ by the Proposition. Hence $\pi^*\mu=\pi^*\nu$ implying $\mu=\nu$. Now letting $f_\mu=a_iv^i+c$ it follows that $\Theta$ sends $\dx f_\mu$ to $\dx c$.

The induced map is surjective since $\Theta(\pi^*\gamma)=\gamma$, for any $\gamma\in\Omega^1_M$. It is injective since if $\Theta(\alpha_\mu)=\gamma=\dx c$, then $h_i-\papa{a_k}{x^i}v^k=\papa{c}{x^i}$ and thus $\alpha_\mu=\dx f_\mu$ with obvious $f_\mu$.
\end{proof}


\subsection{The $\dx_B$ cohomology}

A more plain cohomology associated to the tangent manifold emerges through the theory of Fr\"olicher-Nijenhuis on the graded algebra $\Der_*(\Omega^*_\calm)$ of derivations of $\Omega^*_\calm$, where $\calm$ is any given manifold. Let us recall this extraordinary theory in brief terms, which brings a most complete insight into $\Der_*(\Omega^*_\calm)$. We follow the reference \cite{Michor} quite closely. First, a derivation $D:\Omega^*\lrr\Omega^{*+k}$ is a linear operator such that $D(\alpha\wedge\beta)=(D\alpha)\wedge\beta+(-1)^{kl}\alpha\wedge D\beta$, for every $\alpha\in\Omega^l$ and every $\beta$ differential form. The graded Lie bracket on $\Der_*(\Omega^*_\calm)$ is given by $[D_1,D_2]:=D_1\circ D_2-(-1)^{k_1k_2}D_2\circ D_1$. It satisfies the identities
\[  [D_1,D_2]=-(-1)^{k_1k_2}[D_2,D_1] \]
\[  [D_1,[D_2,D_3]]=[[D_1,D_2],D_3]+(-1)^{k_1k_2}[D_2,[D_1,D_3]]  .  \]
In other words, $[D_1,\ ]$ is itself a derivation of degree $k_1$.

Fundamental derivations are given by $X\lrcorner(\,\cdot\,)=\ip_X(\,\cdot\,)$, for $X\in\XIS_\calm$, of degree $-1$; or given by exterior derivative $\dx$, of degree $+1$, or by $\call_X=[\ip_X,\dx]=\ip_X\dx+\dx\ip_X$ of degree 0.

The algebraic derivations $\ip_X$ can be generalized. Firstly, they are algebraic because they vanish on functions, and therefore are tensorial. Now if $K\in\wedge^{k+1}T^*\calm\otimes T\calm$ or say if $K\in\Omega^{k+1}_\calm(T\calm)$, if $\omega\in\Omega^p_\calm$ and if $X_1,\ldots,X_{k+p}\in\XIS_\calm$, then
\begin{equation}
\ip_K\omega(X_1,\ldots,X_{k+p})=\frac{1}{(k+1)!(p-1)!}\sum_{\sigma\in \mathrm{S}_{k+p}}\sg{\sigma}\,\omega(K(X_{\sigma_1},\ldots,X_{\sigma_{k+1}}),X_{\sigma_{k+2}},\ldots,X_{\sigma_{k+p}})
\end{equation}
defines an algebraic derivation of degree $k$. This important theorem states moreover that any algebraic derivation is of that type: it arises from a vector-valued $(k+1)$-form $K$.

In relation to the previous section we find a simple coincidence. For any $\omega\in\Omega^p_\calm$ and any $K\in\Omega^{1}_\calm(T\calm)$
\begin{equation}
 \ip_K\omega=\frac{1}{(p-1)!}\omega\circ(K\wedge1^{\wedge,(p-1)}).
\end{equation}
In particular, $\ip_{1_{T\calm}}\omega=p\omega$.

Next we recall the Lie derivations, defined by $\call_K=[\ip_K,\dx]=\ip_K\circ\dx-(-1)^{k-1}\dx\circ\ip_K$ for any $K\in\Omega^k_\calm(T\calm)$ --- it is a degree $k$ derivation. It follows easily that $[\call_K,\dx]=0$. And that $K\mapsto \call_K$ is injective.

Another striking theorem shows that every derivation is the sum of an algebraic and a Lie derivation: for each $D$ there exist unique $K$ of degree $k$ and $L$ of degree $(k+1)$ such that
\begin{equation}
D=\call_K+\ip_L .
\end{equation}
Due to the above, $D$ is Lie if and only if $L=0$ if and only if $[D,\dx]=0$. And $D$ is algebraic if and only if $K=0$.

For $K,L\in\Omega^*_\calm(T\calm)$ of degrees $k,l$ respectively, clearly $[[\call_K,\call_L],\dx]=0$, so we have a Lie derivation and therefore there exists unique $(k+l)$-form $[K,L]^{\mathrm{FN}}$ such that $[\call_K,\call_L]=\call_{[K,L]^{\mathrm{FN}}}$. Such bracket is the Fr\"olicher-Nijenhuis bracket of $\Omega^*(T\calm)$ generating a Lie graded subalgebra inside $\Der(\Omega^*)$ via the monomorphism $K\mapsto\call_K$.

For $X,Y\in\XIS_\calm=\Omega^0(T\calm)$, the Fr\"olicher-Nijenhuis
bracket is the usual Lie bracket. With $1\in\Omega^1_\calm(T\calm)$, we find $\call_1=[\ip_1,\dx]=(p+1)\dx-p\dx=\dx$. Also $\call_{[1,K]^{\mathrm{FN}}}=[\dx,\call_K]=0$, and thus $[1,K]^{\mathrm{FN}}=0$.

Needless to say, there are elaborate developments on the computation of the Fr\"olicher-Nijenhuis bracket which are very useful, cf. \cite{Michor}.

For $k$ odd we have
\[ \call_K^2=\frac{1}{2}(\call_K^2-(-1)^{k^2}\call_K^2)=\frac{1}{2}[\call_K,\call_K]=\frac{1}{2}\call_{[K,K]^{\mathrm{FN}}} \]
and so we obtain a coboundary operator if and only if $[K,K]^{\mathrm{FN}}=0$.

We stop here the presentation of the beautiful theory of derivations, in order to focus on our main subject which, we recall, is given by a manifold $M$ and its mirror map $B\in\Omega^1_{TM}(TTM)$.

It has been understood long ago (we use the established notation) that
\[ \dx_B:=\call_B=[\ip_B,\dx]      \]
defines a coboundary operator on $\Omega^*_{TM}$, this is, a derivation on the algebra of differential forms on the tangent manifold with vanishing square. This was most probably first discovered in \cite{Anona1}. The elaborate ways to compute the Fr\"olicher-Nijenhuis give us
\[ \frac{1}{2}[K,K]^{\mathrm{FN}}(X,Y)=[KX,KY]-K[X,KY]-K[X,KY]+K^2[X,Y]  \]
and so it is very easy to see $[B,B]^{\mathrm{FN}}=0$. We also have a straightforward proof.
\begin{prop}
 $\dx_B\dx_B=0$.
\end{prop}
\begin{proof}
 Let $\dx_B\dx_B=(\ip_B\dx-\dx\ip_B)^2=\ip_B\dx\ip_B\dx-\dx\ip_B\ip_B\dx+\dx\ip_B\dx\ip_B$. A simple analyis on 1-forms $\dx x^i,\dx v^j$ and any $f\in\Omega^0$ shows their image under the derivation $\dx_B\dx_B=\frac{1}{2}[\dx_B,\dx_B]$ vanishes. Since they generate the graded algebra, the result follows.
\end{proof}

Indeed it is worth noticing that, for any chart $(x^i,v^j)$ and any function $f$ on $TM$,
\begin{equation}
    \dx_Bf=\papa{f}{v^i}\dx x^i \qquad\ \  \dx_B\dx x^i=0\qquad\ \ \dx_B\dx v^i=0 .
\end{equation}
In particular, $\dx_Bx^i=0$ and $\dx_Bv^i=\dx x^i$. We observe moreover that $[\dx,\dx_B]=0$ or
\[  \dx\dx_B=-\dx_B\dx . \]
\begin{prop}
For every functions $\epsilon_1,\epsilon_2$ on $TM$ we have a coboundary operator ($D\circ D=0$)
 \begin{equation}
  D=\epsilon_1\dx+\epsilon_2\dx_B=\call_{\epsilon_1 1+\epsilon_2 B}:\Omega^*_{TM}\lrr\Omega^{*+1}_{TM}
\end{equation}
if and only if $\epsilon_1,\epsilon_2\in\R$.
\end{prop}
The proof of the proposition is trivial evaluating $D^2$ on the coordinate functions.

When $\epsilon_2=0$ we have essentially de\,Rham cohomology. So we proceed with the study of the derivation $D=c\dx+\dx_B$, where $c$ is a constant. The theory gives rise to a cohomology ring $H^*_D(TM)=\oplus_{q=0}^{2m}H^q_D(TM)$ where
\[  H^q_D(TM)= \frac{\{\omega\in\Omega^q_{TM}: \ D\omega=0\}}{\{D\tau:\ \,\tau\in\Omega^{q-1}_{TM}\}} ,  \]
and also a new Bott-Chern type cohomology, for $q\geq2$,
\[  H^q_{\dx\dx_B}(TM)=\frac{\{\omega\in\Omega^q_{TM}: \ \dx\omega=\dx_B\omega=0\}}{\{\dx\dx_B\tau:\ \,\tau\in\Omega^{q-2}_{TM}\}} . \]
Of course the true geometrical significance of these objects is still to be found.

Notice $\dx_B$ only involves $\papa{}{v^i}$ derivatives, so for every $\omega\in\Omega^p_M$ we have
\begin{equation}\label{pullbacksaredBclosed}
   \dx_B\pi^*\omega=0 .
 \end{equation}
 \begin{prop}
 For any $\omega\in\Omega^p_M$, we have:
 \begin{enumerate}[label=(\roman*)]
  \item $\pi^*\omega$ vanishes in $\dx_B$-cohomology: $\dx_B(\xi\lrcorner\widetilde{\omega})= \pi^*\omega$
  \item $\dx_B\widetilde{\omega}=\dx\pi^*\omega$.

  \item If $\omega$ is closed, then $\dx(\xi\lrcorner\widetilde{\omega})=\widetilde\omega$.
 \end{enumerate}
 \end{prop}
\begin{proof}
 Letting $\omega=o_{i_1\cdots i_p}\dx x^{i_1}\wedge\cdots\wedge\dx x^{i_p}$ with $i_1<i_2<\cdots<i_p$, we recall
 \begin{equation*}
  \begin{split}
    \widetilde{\omega} =v^j\papa{o_{i_1\cdots i_p}}{x^j}\dx x^{i_1}\wedge\cdots\wedge\dx x^{i_p}+ \hspace{53mm} \\
    o_{i_1\cdots i_p}(\dx v^{i_1}\wedge\dx x^{i_2}\wedge\cdots\wedge\dx x^{i_p}+\dx x^{i_1}\wedge\dx v^{i_2}\wedge\cdots\wedge\dx x^{i_p}+\cdots+\dx x^{i_1}\wedge\cdots\wedge\dx x^{i_{p-1}}\wedge\dx v^{i_p})
  \end{split}
 \end{equation*}
 We get immediatley $\dx_B\widetilde{\omega}= \pi^*\dx\omega=\dx\pi^*\omega$. Moreover, we see
 \[ \xi\lrcorner\widetilde\omega = o_{i_1\cdots i_p}(v^{i_1}\dx x^{i_2}\wedge\cdots\wedge\dx x^{i_p}-v^{i_2}\dx x^{i_1}\wedge\cdots\wedge\dx x^{i_p}+\cdots+(-1)^{p+1}v^{i_p}\dx x^{i_1}\wedge\cdots\wedge\dx x^{i_{p-1}}) \]
 and the first part follows too: $\dx_B(\xi\lrcorner\widetilde{\omega})= \pi^*\omega$. Part (iii) was proved in Theorem \ref{teo_Lieisidentitty}.
\end{proof}
Thus $\Omega^p_M\lrr H_{\dx_B}^p(TM)$ through pullback is just the 0 map.

Now the following map is well-defined
\begin{equation}
 H^*(M)\lrr H^*_D(TM) \,,\qquad[\omega]\longmapsto [\pi^*\omega].
\end{equation}
Indeed, if $\omega=\dx\tau$ over $M$, then $D(\frac{1}{c}\pi^*\tau)=\dx\pi^*\tau=\pi^*\omega$ in case $c\neq0$. Further, it is easy to see that $H^0_D(TM)=\{f\in\Omega^0:\ Df=0\}$ is the set of functions `constant along the fibre' if $c=0$ and globally constant if $c\neq0$. Thus $H^0(M)\hookrightarrow H^0_D(TM)$ is injective. Also, the map $H^1(M)\hookrightarrow H^1_D(TM)$ is injective for any $c$, as follows from a simple computation.

Now let us see the complete lifts.
\begin{prop}
 For $c\neq0$, the following map is well-defined
\begin{equation}
 H^*(M)\lrr H^*_D(TM) \,,\qquad[\omega]\longmapsto [\widetilde\omega].
\end{equation}
Moreover, for $p\geq1$, we have $[\widetilde\omega]=0\Leftrightarrow[\pi^*\omega]=0$.
 \end{prop}
 \begin{proof}
  We have $D\widetilde\omega=(c\dx+\dx_B)\widetilde\omega=c\widetilde{\dx\omega}
+\pi^*\dx\omega$. If $\omega$ is closed, then also $D\widetilde\omega=0$. Reciprocally, if $D\widetilde\omega=0$ then we deduce $\dx\omega=0$ . Now, if $\omega=\dx\alpha$, then
  \[  D(\frac{1}{c}\widetilde\alpha-\frac{1}{c^2}\pi^*\alpha)=\dx\widetilde\alpha+\frac{1}{c}\pi^*\dx\alpha-\frac{1}{c}\dx\pi^*\alpha=\widetilde\omega   \]
 so the map is well-defined. The equation $D\widetilde\omega=0$ implies $\dx\omega=0$ (for any $c$).
 Then $\widetilde\omega=D\gamma$ for some $\gamma\in\Omega_{TM}^{p-1}$ iff $\widetilde\omega= D(-\frac{1}{c}\mu+\frac{1}{c}\xi\lrcorner\widetilde\omega)$ for some $\mu\in\Omega_{TM}^{p-1}$. Since $\dx\omega=0$, we have
 \begin{eqnarray*}
   c\widetilde\omega&=& D(-\mu+\xi\lrcorner\widetilde\omega) \\
   &=&-D\mu+c\dx\xi\lrcorner\widetilde\omega+\dx_B\xi\lrcorner\widetilde\omega \\
   &=&-D\mu+c\widetilde\omega+\pi^*\omega
 \end{eqnarray*}
 or $\pi^*\omega=D\mu$.
 \end{proof}

A constant $f$ yields $\widetilde f=0$ or the vanishing of the previous map in degree 0.

The pullback and the complete lift do not have the same range inside $D$-cohomology. In the above, $c$ has pivotal role, as we see through a simple example. Let $\omega=\dx x^1$. Clearly $\dx_B\pi^*\omega=\dx_B\dx x^1=0$ while $\dx x^1=\dx_Bv^1$. On the other hand, $\widetilde\omega=\dx v^1$ and so $\widetilde\omega=\dx_B g=\partial_{v^i}g\,\dx x^i$ has no solution $g$.

We have another simple result worth noticing.
\begin{teo}
 For $q\geq2$, the following map is well-defined
 \begin{equation}
 H^q(M)\lrr H^q_{\dx\dx_B}(TM) \,,\qquad[\omega]\longmapsto [\pi^*\omega].
\end{equation}
\end{teo}
\begin{proof}
 Every pullback of a $q$-form onto $TM$ is $\dx_B$-closed, as seen in \eqref{pullbacksaredBclosed}.
 Now if $\omega=\dx\beta$, then $\dx\dx_B(\xi\lrcorner\widetilde\beta)=\dx\pi^*\beta=\pi^*\omega$.
\end{proof}
We could not decide on the injectivity of the map.

\subsection{On semi-basic forms}

The cohomology ring $H^*_{\dx_B}(TM)$ will not be easy to compute with the instruments so far. The complex defined by $\dx_B$ was first introduced in \cite{Klein}.

In a first path, the cohomology $H^p(TM,\calf)$ with coefficients in the sheaf of rings $\calf\lrr TM$, of germs of smooth functions whose restriction to the fibres of $\pi$ are constant, has been considered. With two separate proofs in \cite{Anona1,AnonaRatovoarimanana,LehmannLejeune}, a fine resolution of $\calf$ with the same complex operator $\dx_B$  has been discovered:
\begin{equation}\label{fineresolutionofF}
0\lrr\calf\lrr\calf^0_0\lrr\calf^1_0\lrr\cdots\lrr\calf^p_0\lrr\calf^{p+1}_0\lrr\cdots.
\end{equation}
Let us recall the definition of the $\calf^i_0$. Clearly $\calf=\ker\dx_B:\cinf{}\lrr\Omega^1$ and $\calf^0_0=\cinf{}$, leading us to consider the \textit{semi-basic} $p$-forms, ie. forms $\omega$ on the tangent manifold such that $Y\lrcorner\omega=0,\ \forall Y\in \ker\dx\pi=V$. 
In local coordinates, we see by (\ref{trantangent1}) that $\omega$ is written as
\[  \omega=o_{i_1\cdots i_p}\dx x^{i_1}\wedge\cdots\wedge\dx x^{i_p}  \]
where $i_1<i_2<\cdots<i_p$ and the coefficients depend on $(x,v)$. It follows that $\dx_B\omega$ is also semi-basic. Next we let $\calf^p_0$ denote the sheaf of germs of semi-basic $p$-forms; its sections are the semi-basic forms.

\begin{teo}[\cite{Anona1,AnonaRatovoarimanana,LehmannLejeune}]
 \eqref{fineresolutionofF} is a fine resolution of $\calf$. In particular,
 \[  H^p(TM,\calf)= \frac{\{\omega\in\Gamma(TM;\calf^p_0): \ \dx_B\omega=0\}}{\{\dx_B\tau:\ \,\tau\in\Gamma(TM;\calf^{p-1}_0)\}} .  \]
\end{teo}
\begin{proof}
Heuristic it may be, this proof does not reproduce exactly those from the references. We wish to see that $\dx_B\omega=\papa{o_{i_1\cdots i_p}}{v^i}\dx x^i\wedge\dx x^{i_1}\wedge\cdots\wedge\dx x^{i_p}=0$ implies $\omega$ is $\dx_B$-exact locally. Certainly one obtains
 \[  \papa{o_{i_1\cdots i_p}}{v^i}\dx v^i\wedge\dx v^{i_1}\wedge\cdots\wedge\dx v^{i_p}=0 , \]
which is the equation of a closed $p$-form on each fibre $\R^m$, and therefore an exact form $\mathrm{d}_{|T_xM}\alpha_x$ in the $v^i$ by Poincaré lemma. Henceforth, $\omega=\frac{1}{p!}\mathrm{d}_{|T_xM}\alpha_x\circ B^{\wedge,p}$. Moreover, we may admit the $\alpha$ to be smooth in $x$. By construction the $\alpha$ are  semi-basic and therefore induce a well-defined global semi-basic $p$-form on $TM$. Of course the equation seen is $\omega=\dx_B\alpha$.
\end{proof}

One can define a \textit{$k$-semi-basic} $p$-form to be a $p$-form $\omega$ such that $V_1\wedge\cdots\wedge V_{k+1}\lrcorner\omega=0,\ \forall V_1,\ldots,V_{k+1}\in V\subset TTM$. When $k\geq p$ or $k\geq m$, then $\omega$ is just any $p$-form.

Let $\calf^p_k$ be the sheaf of germs of $k$-semi-basic $p$-forms. In particular, $\calf^p_0$ are the usual semi-basic. It follows that $\calf^p_0\subset\calf^p_{1}\subset\calf^p_2\cdots\subset\calf^p_p=\Omega^p$.
Clearly $\dx_B:\calf^p_k\lrr\calf^{p+1}_k$ is a coboundary operator. So we may consider for each $k$ the complex
\begin{equation}\label{fineresolutionofF^p_k?}
0\lrr\ker\dx_B\lrr\calf^k_k=\Omega^k\lrr\calf^{k+1}_k\lrr\cdots\lrr\calf^p_k\lrr\calf^{p+1}_k\lrr\cdots .
\end{equation}
Another complex may be written using the exact sequence $0\rr\calf^p_i\rr\calf^p_{i+1}\rr\calq^p_{i+1}\rr0$.


\vspace{3mm}

\begin{center}
 Statements and Declarations
\end{center}


\vspace*{1mm}

\noindent Funding: The research leading to these results has received funding from Funda\c c\~ao para a Ci\^encia e a Tecnologia. Project Ref. UIDB/04674/2020.


\vspace*{5mm}

\textsc{R. Albuquerque}\ \ \ \textbar\ \ \
{\texttt{rpa@uevora.pt}}

Centro de Investiga\c c\~ao em Mate\-m\'a\-ti\-ca e Aplica\c c\~oes

Rua Rom\~ao Ramalho, 59, 671-7000 \'Evora, Portugal


\end{document}